\documentclass[11pt,reqno]{amsart}
\usepackage{latexsym,amsmath,amsfonts,amssymb,amsthm,textcase}
\usepackage{stmaryrd}
\usepackage{soul}
\usepackage{color, xcolor}
\usepackage{hyperref}

\usepackage{hyperref}
\usepackage{enumerate}
\usepackage{url}
\usepackage{enumerate}
\usepackage{verbatim}
\theoremstyle{plain}
\makeatletter
\@namedef{subjclassname@2020}{%
\textup{2020} Mathematics Subject Classification}
\makeatother

\def\Z{\mathbb Z}
\def\C{\mathbb C}
\def\N{\mathbb N}
\def\bN{\mathbf N}

\def\T{\mathbb T}
\def\cR{{\mathcal R}}

\def\cP{{\mathcal P}}

\def\1{{\bf 1}}

\def\supp{{\mathrm {supp}\;}}

\def \meas {\mathrm{meas}}

\def\cA{{\mathcal A}}\def\cB{{\mathcal B}}

\def\cP{{\mathcal P}}
\def\cR{{\mathcal R}}

\def\bC{{\mathbb C}}

\def\bN{{\mathbb N}}
\def\bQ{{\mathbb Q}}\def\bR{{\mathbb R}}\def\bT{{\mathbb T}}

\def\bZ{{\mathbb Z}}

\def \bc {\mathbf c}

\def \bf {\mathbf f}

\def \bx {\mathbf x}

\def\fn{{\mathfrak n}}

\def\alp{{\alpha}}

\def\del{{\delta}}

\def\tet{{\theta}}

\def\sig{{\sigma}}

\def\eps{\varepsilon}

\def\le{\leqslant} \def\ge{\geqslant}

\def\d{{\,{\rm d}}}

\def \sig{{\sigma}}

\theoremstyle{plain}
\newtheorem{theorem}{Theorem}[section]

\newtheorem{proposition}[theorem]{Proposition}

\newtheorem{lemma}[theorem]{Lemma}

\theoremstyle{remark}
\newtheorem{remark}{Remark}[section]
\numberwithin{equation}{section}
\begin{document}
\title[Roth-type Theorem for high-power system in Piatetski-Shapiro primes]{R\MakeTextLowercase{oth-type} T\MakeTextLowercase{heorem for high-power system in} P\MakeTextLowercase{iatetski-}S\MakeTextLowercase{hapiro primes} (\uppercase\expandafter{\romannumeral2})}
\author{X\MakeTextLowercase{iumin} R\MakeTextLowercase{en}, Y\MakeTextLowercase{u-chen} S\MakeTextLowercase{un}, Q\MakeTextLowercase{ingqing} Z\MakeTextLowercase{hang}, R\MakeTextLowercase{ui} Z\MakeTextLowercase{hang}}
\address{School of Mathematics, Shandong University, Jinan, Shandong 250100, People's Republic of China}
\email{xmren@sdu.edu.cn}
\address {Department of Mathematics and Statistics, University of Turku, 20014 Turku, Finland}
\email{yuchensun93@163.com}
\address{School of Mathematics, Shandong University, Jinan, Shandong 250100, People's Republic of China}
\email{qingqingzhangsdu@126.com}
\address{School of Mathematics, Shandong University, Jinan, Shandong 250100, People's Republic of China}
\email{rzhang@mail.sdu.edu.cn}
\begin{abstract}
   We consider the nonlinear system $c_1p_1^d +c_2p_2^d + \dots + c_s p_s^d = 0$ with  $c_1, c_2,\dots, c_s\in\mathbb Z$ being nonzero and satisfying $c_1 +c_2 + \dots + c_s = 0$. We show that  for $s\ge 2\lfloor \frac{d^2}2\rfloor+1$  and $c\in\left(1, 1+c(d,s)\right)$,   if the system has only $K$-trivial solutions in subset $\cA$ of Piatetski-Shapiro primes up to $x$ and corresponding to $c$,  then $|\cA| \ll \frac{x^{\frac1c}}{\log x} $$\left(\log \log \log \log x\right)^{\frac{2-s}{dc}+\eps}$. 
\end{abstract}
\maketitle

\section{Introduction}

Let $N$ be a large positive integer. Erd\H{o}s-Tur\'{a}n conjectured \cite{Erdo-1936}  that if $A\subset [N]=\{1, 2,\dots ,N\}$ contains no $k$-term arithmetic progressions, then 
$|A|=o(N) $ as $N\rightarrow\infty$. 
Roth \cite{Roth-1953} proved that this conjecture is true when $k=3$ and showed that if $A\subset [N ]$ contains no nontrivial solutions to the
 diophantine equation $x -2y+ z = 0,$ then $|A| \ll N/\log \log N.$ This work was improved by many authors and was generalized to translation-invariant linear equations (see \cite{Heath-Brown-1987, Szemeredi-1990, Bloom-Sisask-2020}, etc). 
 In $2005$, Green \cite{Green-2005} established a remarkable analogue of Roth's theorem in primes which stated that any set containing a positive proportion of the primes contains nontrivial $3$-term arithmetic progressions. 
 In recent years, extensions of Roth's theorem to nonlinear systems have been investigated by many authors. Concerning the homogeneous nonlinear system
\begin{equation}\label{bustard}
c_1x_1^d +c_2x_2^d + \dots + c_s x_s^d = 0
\end{equation}
subject to the condition
\begin{equation}\label{Maincondi}
    c_1 +c_2+ \dots + c_s = 0,
\end{equation}
Browning and Prendiville considered the quadratic case $d=2$ and showed in  \cite{Browning-Prendiville-2017} that  if $s\ge 5$ and \eqref{bustard} contains only $K$-trivial solutions in $A\subset [N]$, then for arbitrary $\varepsilon >0,$
\begin{equation*}\label{full}
|A| \ll N\left(\log \log \log N\right)^{\frac{2-s}{2}+\varepsilon}.
\end{equation*}
Here a solution ${\bx}=(x_{1}, x_{2},\dots, x_{s})$ to \eqref{bustard} is said to be $K$-trivial if $(x_{1}^{d}, x_{2}^{d}, \dots, x_{s}^{d})\in K$, 
 where $K$ is a union of $k$ proper subspaces of the rational hyperplane 
 \begin{equation}\label{MainEqn}
 c_1x_1 +c_2x_2 + \dots + c_s x_s = 0
 \end{equation}
 with each of them containing the diagonal elements $\{(x, x, \ldots,x): x \in \bQ\}.$
Chow \cite{Chow-2017} studied the equation \eqref{bustard} of degree $d\ge 2$ in prime
numbers. Let $ \mathcal {P}_N$  denote the set of all prime numbers up to $N$.
Chow's result states that for $s\ge s(d)$, if the equation \eqref{bustard} has only $K$-trivial solutions in  $\mathcal A\subset  \mathcal P_N$, then
\begin{equation*}\label{prime}
|\mathcal A| \ll \frac {N}{\log N}\left(\log \log \log \log N\right)^{\frac{2-s}{d}+\varepsilon},
\end{equation*}
where 
$$
s(2)=5,\quad s(3)=9,\quad s(4)=15,\quad s(d)=d^2+1~\quad (d\ge 5).
$$

In our previous work \cite{Ren-Zhang-Zhang-2023, Zhang-Zhang-2023}, we have considered the equation \eqref{bustard}  in Piatetski-Shapiro primes for $d\ge 2$. Let $c\in (1, 2)$  and $\bN^c=\{\lfloor n^c \rfloor: n\in\N\},$ where $\lfloor x\rfloor$ is the largest integer not exceeding $x$. The primes in $\N^c$ (denoted by $\mathcal {P}^{c}$) are called Piatetski-Shapiro primes corresponding to $c$.
Let $x$ be a positive integer and write $\mathcal {P}^{c}_x=\cP^c\cap [x].$ 
In 1953, Pyatecki\u{\i}-\v{S}apiro \cite{Piatetski-Shapiro-1953} firstly showed the prime number theorem:  for $1<c<{12}/{11}$,
\begin{equation*}\label{pnt}
|\cP^c\cap [x]|=(1+o(1)) \frac{x^{\frac{1}{c}}}{\log x} \quad \text { as } x \rightarrow \infty.
\end{equation*}
The range of $c$ has been improved many times (see \cite{Heath-Brown-1983, Kolesnik-1985, Liu-Rivat-1992}, etc). The current best known result $1<c<\frac{2817}{2426}$ is due to Rivat and Sargos \cite{Rivat-Sargos-2001}. 

It was proved in   \cite{Ren-Zhang-Zhang-2023} and \cite{Zhang-Zhang-2023} that for $s\ge \bar{s}(d)$ and $c\in\left(1, 1+{c}(d,s)\right)$, if the equation \eqref{bustard} has only $K$-trivial solutions in  $\mathcal A\subset \mathcal P_x^c$, then 
\begin{equation*}
 |\mathcal A| \ll_{\eps} \frac{x^{\frac1c}}{\log x}(\log \log \log \log x)^{\frac{2-s}{dc}+\eps},
\end{equation*}
where 
\begin{equation}\label{1.4}
\bar{s}(d)=2^{d}+3~\ (2\le d\le 4),\quad \bar{s}(d)=d(d+1)+3\quad (d\ge 5),
\end{equation}
and
$${c}(d,s)=\min\left\{\frac{2d}{(\bar{s}(d)-2)\left((3d-4)(9d^2-4)-d\right)-d},
\frac{d}{(\bar{s}(d)-3)s-d}\right\}.
$$ 
The number $\bar{s}(d)$ is determined by the restricted estimate in \cite[Proposition 5.1]{Ren-Zhang-Zhang-2023} and \cite[Proposition 5.1]{Zhang-Zhang-2023}, where we have used Bourgain's strategy in \cite[Section 4]{Bourgain-1989}. 

In this paper, we will borrow the idea in \cite{Sun-Du-Pan-2023} to treat the restricted estimate (see  Lemma \ref{psiupsilon} below). This allows us to  reduce the number of variables in \eqref{1.4} to $\bar{s}(d)=2\lfloor \frac{d^2}2\rfloor+1$. Moreover, for $d\ge 4$  we will improve the admissible range of $c$ in \cite[Lemma 5]{Akbal-Guloglu-2016} by using new exponential sum estimates.  

\begin{theorem}\label{MainThm}
    Let $d\ge2$ and $\bar{s}(d)=2\lfloor \frac{d^2}{2}\rfloor+1$. Suppose that the $d$-th power system
    \begin{equation}\label{equation}
    c_1x_1^d +c_2x_2^d + \dots + c_s x_s^d = 0 \quad {\rm with}\quad  c_1+c_2+\dots+c_s=0
    \end{equation}
    has only $K$-trivial solutions
    in $\mathcal A\subset   \mathcal {P}^{c}_x$. Then for $s\ge \bar{s}(d),$ $c\in\left(1, 1+c(d,s)\right)$ and arbitrary $\varepsilon >0,$ we have
    \begin{align*} \label{DensityBound}
    |\mathcal A| \ll_{\eps} \frac{x^{\frac1c}}{\log x}(\log \log \log \log x)^{\frac{2-s}{dc}+\eps},
    \end{align*}
    where  
    $$
    c(2,s)=    \min\left\{\frac{1}{54}, \frac{1}{2s-1}\right\},\quad 
    c(3,s)=\min\left\{\frac{1}{495}, \frac{3}{8s-3}\right\},
    $$
    and
    \begin{equation}\label{c_0}
    c(d,s)=\begin{cases}
      \min\left\{\frac{2}{(4\bar{s}(d)-3)\left(d(d+1)^{2}-1\right)-1},
    \frac{d}{(\bar{s}(d)-1)s-d}\right\},&  {\rm if}\ 4\le d\le11,\\
    \min\left\{\frac{4}{(4\bar{s}(d)-3)(3(3d-2)(3d+2)-2)-2}, 
    \frac{d}{(\bar{s}(d)-1)s-d}\right\},&{\rm if}\ 2 \mid d,\ d\ge 12, \\
    \min\left\{\frac{4}{(4\bar{s}(d)-3)(3(3d-1)(3d+1)-2)-2},
    \frac{d}{(\bar{s}(d)-1)s-d}\right\}, &{\rm if}\ 2\nmid d,\ d\ge 12.
    \end{cases}
    \end{equation}
    \end{theorem}
\begin{remark} Write
 $
 S(d)=\bar{s}(d)-1.
 $ 
 Then $S(d)$ is chosen to satisfy  the inequality 
\begin{equation}\label{1-zz}
\int_0^1\Big|\sum_{n\le x}e(\alpha n^{d})\Big|^{S(d)}d\alpha\ll x^{S(d)-d}L.
\end{equation}
For  $d=2, 3,$ we can choose  $S(d)=2^d=2\lfloor \frac{d^2}{2}\rfloor$ by applying \cite[Theorem 3]{Blomer-2006} and \cite[Theorem 2]{Vaughan-1986-}.  For $d\ge 4,$ one can follow the proof of \cite[Theorem 4.1]{Wooley-2012} verbatim to show that $S(d)=2\lfloor \frac{d^2}{2}\rfloor$ 
is admissible in \eqref{1-zz}.
 \end{remark}
 \medskip
 
 \begin{remark}   When $d=2,3,4$, by using  the mean value estimation in \cite[Lemma 10]{Akbal-Guloglu-2016} instead of  \eqref{1-zz} in the proof of Lemma \ref{falphaeta}  and \cite[Lemma 7.1]{Zhang-Zhang-2023}, we can extend the range of $c$ in Theorem \ref{MainThm} slightly  by choosing $c(2,s)=\min\{\frac{1}{28},\frac{1}{s-1}\},  c(3,s)= \min\{\frac{1}{255},\frac{3}{4s-3}\}$ and $c(4,s)=\min\{\frac{1}{1831},\frac{4}{9s-4}\}$, respectively.
    \end{remark}
 \medskip
 
  \begin{remark} 
  One can still apply Harman's sieve method to improve our $c$, see \cite{Sun-Du-Pan-2023} for more details.
  \end{remark}
 \medskip
$\textbf{Notations.}$ We abbreviate $(x_1, x_2, \ \dots,x_s)$ to $\bx$, and if $x$ is a positive integer, we  write $\mathbb [x]=\{1, 2,\dots, x\}$.  For any $ A\subset \mathbb N$, we write
$A_x=A\cap [x]$ and denote by $1_{ A}$ the indicator function of $A$. 
Let $\eps>0$ be an arbitrarily small  number and its values may differ between instances.
We write $\bT$ for the torus $\bR / \bZ$, and will identify $\bT$ with the unit interval $[0,1)$ throughout this paper.
For large positive number $x$, we  write $L = \log x.$
\bigskip

\section{Proof of Theorem \ref{MainThm}}
Let $f: \bZ \to \bC$. Define $ \| f \|_\infty = \sup_{n}|f(n)|$ and the  $L^q$ norm
$$\| f \|_q = \Bigl( \sum_n |f(n)|^q \Bigr)^{1/q}.$$
Here $q>0$ may not be an integer. If $\| f\|_1 < \infty$, then the Fourier transform of $f$   is  defined as 
$$
\widehat{f}(\alp)= \sum_{n\in \mathbb Z} f(n) e(\alp n),\quad \alpha\in\bT.
$$  
For functions on $\bT$, $L^q$  norm is taken with respect to the Haar probalility measure.
Following \cite{Browning-Prendiville-2017},  a majorant on $[N]$ is defined as  a non-negative function $\nu : \mathbb{Z}\rightarrow [0,\infty)$ with support contained in $[N]$.



To prove  Theorem \ref{MainThm}, we will use the following transference principle in  Browning and Prendiville  \cite[Proposition 2.8]{Browning-Prendiville-2017}.  
\begin{proposition}\label{Browning-Prendiville}
Let $s\ge 3$ and  $c_i\  (i=1, 2, \dots, s)$ be non-zero integers satisfying $c_1+c_2+\cdots+c_s=0$. Let $K$ be a union of $k$ proper subspaces of the hyperplane \eqref{MainEqn}. Suppose that $\nu$ is a majorant on $[N]$ satisfying
\begin{itemize}
  \item[\tiny$\bullet$] {\normalsize {\rm (Fourier decay)}~$\| \widehat{\nu} - \widehat{1}_{ [N]} \|_\infty \le \tet N$ for $\theta\in (0,1]$;}
      
  \vspace{1pt}
  \item[\tiny$\bullet$] {\normalsize {\rm (Restriction estimate)}~ $\sup_{|f| \le \nu} \int_\bT |\widehat{f}(\alp)|^u \d \alp \ll_u \| \nu \|_1^u N^{-1}$ for some $u\in [s-1,s)$;}
      
   \vspace{1pt}
  \item[\tiny$\bullet$] {\normalsize  {\rm (K-trivial saving)}~$\sum_{(x_1,\dots,x_s) \in K} \prod_{i=1}^s \nu(x_i) \ll_{k,s} \| \nu \|_1^s N^{-1-\eta}$ for some $\eta>0$.}
\end{itemize}
 \vspace{1pt}
Then for any  $ A\subset \supp(\nu),$ if $A$  contains only K-trivial solutions to  \eqref{MainEqn}, then
\begin{align*}
\sum_{n\in A} \nu(n)\ll\frac{N}{\min\{\log\log(1/\theta),\log N\}^{s-2-\eps}},
\end{align*}
where the implied constant depends at most on $\bc, u, \eta, k $ and $ \eps$.
\end{proposition}


 We begin by defining the  majorant $\nu$. Let $x$ be a large positive number and write
$w = \frac12 \log \log x$ and $ W= 4d^3 \prod_{p \le w} p.$
It follows by the prime number theorem that
$W\asymp e^{w} =\sqrt{ \log x}.$
Let $b\in[W]$ with 
$
-b\in \left(\bZ/W\bZ\right)^{\times d}=\{z^{d}: z\in \mathbb (\bZ/W\bZ)^{\times}\},
$ 
and write
$$
\sig(b)=\left|\left\{z\in[W]:\ z^d\equiv-b\ {\rm mod }\ W \right\}\right|.
$$
 We define the majorant function $\nu_{b}:\mathbb N\to \mathbb R$ as follows:
\begin{align*} \label{nudef}
\nu_b\left(n\right) =
\begin{cases}
\frac {c\varphi(W)}{ \sig(b) W} p^{d-\frac{1}{c}} \log p, & \text{if } Wn-b = p^d \text{ for some } p \in \cP^c_{x}, \\
0, &\text{if \ otherwise}.
\end{cases}
\end{align*}
 For $\cA\subset \cP^c_{x}$, we write
\begin{equation} \label{cAdef}
\cB(b) = \{ n \in \bZ: Wn-b = p^d \text{ for some } p \in \cA \}.
\end{equation}
It is easy to see that
\begin{equation} \label{Ddef}
\cB(b)\subset \supp(\nu_b)\subset [N]\quad {\rm  where}\quad N =\lfloor x^d/W\rfloor + 1.
\end{equation}

\begin{lemma} [Density transfer] \label{DT}  Let $c\in (1, 1+\frac{391}{2426})$ and $\cA\subset \cP^c_{x}.$
    Then there exists $ b\in [W]$ with $-b\in (\bZ/W\bZ)^{\times d}$ such that  for $\cB=\cB(b), \nu=\nu_{b},$ there holds
    \begin{equation} \label{DensityTransfer}
    \sum_{n \in \cB} \nu\left(n\right) \gg \delta^d N\quad {\rm with}\quad
    \delta=|\cA|^c\frac{\log^cx}{x}.
    \end{equation}
\end{lemma}
\begin{proof}
    See \cite[Lemma 2.6]{Ren-Zhang-Zhang-2023} and \cite[Lemma 2.6]{Zhang-Zhang-2023}.
\end{proof}
\begin{lemma}[Fourier decay]\label{Fourierdecay}
    Let $\nu$ be defined as in {\rm Lemma \ref{DT}.} Then for $c \in(1, 1+c_1(d))$, we have
\begin{equation*}
\|\widehat\nu-\widehat{1}_{[N]}\|_{\infty}\ll w^{\eps-\frac12}N,
\end{equation*}
where $c_1(2)=\frac{7}{75}, c_1(3)=\frac{3}{77},$ and for $d\ge 4,$
\begin{align*}\label{c1d}
c_1(d)=\begin{cases}
\min\left\{h(d), k(d)\right\},    &{\rm if}\ 2 \mid d,\\
\min\left\{h(d), l(d)\right\},    &{\rm if}\ 2\nmid d\\
\end{cases}
\end{align*}
with 
\begin{equation}\label{hkl}
h(d)=\frac{1}{d(d+1)^2-1}, \quad  k(d)=\frac{2}{27d^2-14},\quad  l(d)=\frac{2}{27d^2-5}.
\end{equation}
\end{lemma}
\begin{proof}
    See \cite[Proposition 4.1]{Ren-Zhang-Zhang-2023} and \cite[Proposition 4.1]{Zhang-Zhang-2023}.
\end{proof}

\begin{lemma}[Restriction estimate]\label{psiupsilon}
    Let $f:\,\Z\to\C$ be an arithmetic function which satisfies $|f|\le\nu$ with $\nu$ the majorant being defined in  {\rm Lemma \ref{DT}.} 
    Then for any $u>S(d)(1+\frac{2(c-1)}{1-\theta(d,c)})$ where $S(d)=2\lfloor \frac{d^2}2\rfloor$ and $c\in\left(1, 1+c_2(d)\right)$, we have
    \begin{align*}
    \int_{\bT}\left|\widehat{f}\left(\alpha\right)\right|^ud\alpha\ll_u N^{u-1},
    \end{align*}
  where $\theta(2,c)=\frac{4c+7}{13},\  \theta(3,c)=\frac{15c+14}{30}, $
  \begin{align*}
\theta(d,c)&=\frac{1+c}2\cdot\begin{cases}
(1-h(d)), \ & {\rm if}\  4\le d\le 11,\\
(1-k(d)), & {\rm if}\ 2 \mid d,\ d\ge 12,\\
(1-l(d)), & {\rm if}\ 2\nmid d,\ d\ge 12;
\end{cases}
    \end{align*}
   and $c_2(2) =   \frac{1}{54}, \ c_2(3)= \frac{1}{495}, $ 
    \begin{align*}
    c_2(d)&=\begin{cases}
   \frac{2h(d)}{4S(d)+1-h(d)}, & {\rm if}\ 4\le d\le 11, \\
    \frac{2k(d)}{4S(d)+1-k(d)}, &{\rm if} \ 2 \mid d,\  d\ge 12,\\
    \frac{2l(d)}{4S(d)+1-l(d)}, &{\rm if} \ 2\nmid d,\ d\ge 12.
    \end{cases}
    \end{align*}
\end{lemma}
We will prove Lemma \ref{psiupsilon} by using  Bourgain's strategy and  Weyl sum estimates.  Note that $0<\frac{2S(d)(c-1)}{1-\theta(d,c)}<1$ for $c\in\left(1, 1+c_2(d)\right)$.  
Thus $\bar{s}(d)=S(d)+1$ can  be obtained in Theorem \ref{MainThm}.  However, the range of $c$ is worse than that in  \cite[ Proposition 5.1]{Ren-Zhang-Zhang-2023} and \cite[Proposition 5.1]{Zhang-Zhang-2023}. We will give the proof of Lemma \ref{psiupsilon} in  Section \ref{Restriction}.

\begin{lemma}[K-trivial saving]\label{k-trivial}
    Let $\nu$ be defined in {\rm Lemma \ref{DT}}  and $S(d)$ be as in {\rm Lemma \ref{psiupsilon}}.  Then for $s\ge S(d)+1$ and $c\in(1, 1+\frac{d}{sS(d)-d})$,  there holds
     \[
 \sum_{(x_1,\dots,x_s) \in K} \prod_{i=1}^s \nu(x_i) \ll_{k,s,\eta} \| \nu \|_1^s N^{-1-\eta},  
 \]
where $\eta=\frac{dc-s(c-1)S(d)}{dc(s-1)}-\eps>0$.
\end{lemma}
\begin{proof}
When $d\ge 3$, one is referred to \cite[Corollary 7.2]{Zhang-Zhang-2023}; when $d=2$, the proof is similar.
\end{proof}

\begin{proof}[Proof of Theorem \ref{MainThm}]

We will use Proposition \ref{Browning-Prendiville}
 to prove Theorem \ref{MainThm}.  It is easy to see that 
 $
 c_2(d)\le c_1(d)\le \frac{391}{2426}.
 $
Let
 $$
 c(d,s)=\min\left\{c_2(d), \frac{d}{sS(d)-d}\right\}.
 $$
 One can check that  $c(d,s)$ can be  expressed as \eqref{c_0} where $\bar{s}(d)=S(d)+1$.  
Then for $c\in (1,1+c(d,s)),$ the majorant $\nu$ defined in Lemma \ref{DT} satisfies  Lemmas \ref{Fourierdecay}--\ref{k-trivial}. 

    Let  $\cA\subset \cP_x^c$  and  suppose that \eqref{equation} has only $K$-trivial solutions in $\cA.$
    Let $\cB$ be the lifting defined as in Lemma \ref{DT}.
    We will apply Proposition \ref{Browning-Prendiville} to the set $A=\cB.$
    We first show that  \eqref{MainEqn} has only $K$-trivial solutions in $\cB,$ which means that  if $\mathbf{n} \in \mathcal{B}^s$ and $\mathbf{c} \cdot \mathbf{n} = 0,$ then $\mathbf{n} \in K.$
    Now suppose $\mathbf{n}=(n_1,n_2, \dots, n_s) \in \mathcal{B}^s$ satisfies $\mathbf{c} \cdot \mathbf{n} = 0$. Then  \eqref{Maincondi} implies  $\mathbf{c} \cdot (W\mathbf{n}-b\mathbf1 ) = 0.$ By \eqref{cAdef}, for each $1 \le i \le s,$ there exists $p_i\in \cA$ such that $Wn_i - b=p_i^d.$ Thus
    $$
    c_1p_1^d+c_2p_2^d+ \dots +c_sp_s^d=\mathbf{c} \cdot (W\mathbf{n}-b\mathbf1 ) = 0.
    $$
    Since \eqref{equation} has only $K$-trivial solutions in $\cA,$ we have $W\mathbf{n}-b\mathbf1=(p_1^d, p_2^d,\dots,p_s^d)\in K.$
    Since $\mathbf1\in K$ and $K$ is invariant under translations and dilations,  we get $\mathbf{n} \in K$, which confirms that $\mathcal{B}$ has only $K$-trivial solutions to \eqref{MainEqn}.
    Note that $\mathcal{B} \subseteq \supp\left(\nu\right)\subset [N].$ Incorporating these facts into Proposition \ref{Browning-Prendiville}, we deduce that
    \[
    \sum_{n \in \mathcal{B}} \nu\left(n\right) \ll \frac N{ \min \left\{ \log \log \left(w^{\frac12-\eps}\right), \log N \right\}^{s-2-\eps}}
     \ll N\left(\log \log \log \log x\right)^{2-s+\eps}.
    \]
   This together with \eqref{DensityTransfer} yields
$$
    |\cA|^{c} \log^{c} x\ll x\left(\log \log \log \log x\right)^{\frac{2-s}{d}+\eps}.
$$
 \end{proof}

\bigskip

\section{Preliminaries for the proof of Lemma \ref{psiupsilon}}\label{Restriction}

\begin{lemma}\label{trig-estimate} Let $\psi(x)=x-\lfloor x\rfloor-\frac12$ and
 $H>1.$ Then there exists a trigonometric polynomial
$$
\psi^*(x)=\sum_{1 \leqslant|h| \leqslant H} a_h e(h x) \quad\left(a_h \ll|h|^{-1}\right)
$$
such that for any real number $x$,
$$
\left|\psi(x)-\psi^*(x)\right| \leqslant \sum_{|h|<H} b_h e(h x) \quad\left(b_h \ll H^{-1}\right).
$$
\end{lemma}
\begin{proof}
See \cite[Appendix]{Graham-Kolesnik-1991}.
\end{proof}

\begin{lemma}\label{estimation-1}
 Let $N$ be a large parameter, $N \le N_{1}\le 2N$, and  $k \geqslant 3$ be an integer. Suppose that $f\in [N, N_{1}] \rightarrow \mathbb{R}$ has $k$-th  order continuous derivatives and satisfies 
 $$
 0<\lambda_k \leqslant |f^{(k)}(t)| \leqslant A \lambda_k,\quad {\rm for}\ t \in [N, N_{1}].
 $$ 
 Then one has
\begin{align*}
 \sum_{N<n \leqslant N_{1}} e(f(n))
\ll_{A, k, \varepsilon} N^{1+\varepsilon}\left(\lambda_k^{\frac1{k(k-1)}}+N^{-\frac1{ k(k-1)}}+N^{-\frac2{k(k-1)}} \lambda_k^{-\frac2{ k^2(k-1)}}\right) .
\end{align*}
\end{lemma}
\begin{proof}
See \cite[Theorem 1]{Heath-Brown-2017}.
\end{proof}

\begin{lemma}\label{estimation-2}
 Let $k \geqslant 3$ be an integer, and let $\alpha_1,\alpha_2, \ldots, \alpha_k \in$ $\mathbb{R}$. Suppose that there exists a natural number $j$ with $2 \leqslant j \leqslant k$ such that, for some $a \in \mathbb{Z}$ and $q \in \mathbb{N}$ with $(a, q)=1$, one has $\left|\alpha_j-\frac{a}{q}\right| \leqslant q^{-2}$. Then,
$$
\sum_{1 \leqslant n \leqslant N} e\left(\alpha_1 n+\alpha_2 n+\cdots+\alpha_k n^k\right) \ll N^{1+\varepsilon}\left(q^{-1}+N^{-1}+q N^{-j}\right)^{\frac1{k(k-1)}} .
$$
\end{lemma}
\begin{proof}
See \cite[Theorem 5]{Bourgain-2017}.
\end{proof}

\begin{lemma}\label{x831c}
Let $c\in (1, 1+c_3(d))$ and $\theta(d,c)$ be defined as in {\rm Lemma \ref{psiupsilon}}. Then uniformly for $\theta\in\bT,$ we have
\begin{equation*}
\sum_{m\in\mathbb N^c_x}cm^{d-\frac{1}{c}}\cdot e\left(m^d\theta\right)=
\sum_{m\le x}m^{d-1}\cdot e\left(m^d\theta\right)+O(x^{d-\frac{1-\theta(d,c)}c+\varepsilon}),
\end{equation*}
where  $c_3(2)=\frac12, c_3(3)=\frac{1}{15},$ and
\begin{align*}
c_3(d)&=\begin{cases}
\frac{2h(d)}{1-h(d)},\quad &{\rm if}\ 4\le d\le 11, \\
\frac{2k(d)}{1-k(d)},\quad &{\rm if}\ 2 \mid d,\ d\ge12, \\
\frac{2l(d)}{1-l(d)}, \quad&{\rm if}\ 2\nmid d,\ d\ge 12,
\end{cases}\\
\end{align*}
where $h(d), k(d), l(d)$ are defined as in \eqref{hkl}.
\end{lemma}
\begin{proof}
For  $d=2,~3$, one  is referred to \cite[Lemma 6.1]{Ren-Zhang-Zhang-2023} and \cite[Lemma 6.1]{Zhang-Zhang-2023}, respectively. 

For $4\le d\le 11$, the proof is analogous.  In fact, since $\lfloor-n^{\frac{1}{c}}\rfloor-\lfloor-\left(n+1\right)^{\frac{1}{c}}\rfloor=1$ or $0$ according as $n \in\N^c$ or not, and
$$
\lfloor-n^{\frac{1}{c}}\rfloor=-n^{\frac{1}{c}}-\psi\left(-n^{\frac{1}{c}}\right)-\frac12,
\quad (1+\frac{1}{n})^{\frac{1}{c}}=1+\frac{1}{cn}+O(\frac{1}{n^2}),
$$
we get
\begin{align*}
&\sum_{\substack{ m\in \mathbb{N}^c_x}}cm^{d-\frac{1}{c}}\cdot e\left( m^d\theta\right)\\
&=\sum_{m\le x}c m^{d-\frac{1}{c}}\cdot e\left( m^d\theta\right)\left(\lfloor-m^{\frac{1}{c}}\rfloor-\lfloor-\left(m+1\right)^{\frac{1}{c}}\rfloor\right)\\
&=\sum_{m\le x}m^{d-1}\cdot e\left(m^d\theta\right)+\sum_{m\le x}c m^{d-\frac{1}{c}}\cdot e\left( m^d\theta\right)\Delta \psi(m)+O\left(x^{d-1}\right),
\end{align*}
where
\begin{equation*}
\Delta \psi(x)=\psi\left(-\left(x+1\right)^{\frac{1}{c}}\right)-\psi\left(-x^{\frac{1}{c}}\right).
\end{equation*}
It is easy to check that for $c\in(1, 1+\frac{2h(d)}{1-h(d)})$, we have $0<\theta(d,c)<1$,  thus the above $O$-term is acceptable in Lemma \ref{x831c}.  So it remains to prove
$$
\sum_{m\le x}c m^{d-\frac{1}{c}}\cdot e\left( m^d\theta\right)\Delta \psi(m)\ll x^{d-\frac{1-\theta(d,c)}c+\varepsilon}.
$$
By dyadic subdivision and partial summation, it suffices to show 
\begin{equation}\label{main-estimate}
 \sum_{m \sim y} e\left( m^d\theta\right) \Delta \psi(m)\ll y^{\frac{\theta(d,c)}c+\varepsilon}.
\end{equation}
  Applying Lemma \ref{trig-estimate},  we have (see \cite[Section 4.6]{Graham-Kolesnik-1991})
\begin{equation*}
\sum_{m \sim y} e\left( m^d\theta\right) \Delta \psi(m) \ll A(y)+B(y),
\end{equation*}
where
$$
A(y)=H_y^{-1} \sum_{|h|<H_y}\left|\sum_{m \sim y} e\left(h m^\frac1c\right)\right|
$$
and
$$
B(y)=y^{\frac1c-1} \sum_{1 \leqslant|h| \leqslant H_y} \max _{y<y^{\prime} \leqslant 2 y}\left|\sum_{y<m \leqslant y^{\prime}} e\left( m^d\theta+h m^\frac1c\right)\right| .
$$

Using the exponent pair $(\frac12,\frac12)$ (see \cite[Chapter 3]{Graham-Kolesnik-1991}), we obtain the estimate
$$
\sum_{m \sim y} e\left(h m^\frac1c\right) \ll |h|^{\frac12} y^{\frac1{2c}}+|h|^{-1} y^{1-\frac1c} \quad(h \neq 0) .
$$
Set
\begin{align*}
H=H_y=y^{1-\frac1c+v},
\end{align*}
where $0<v<1$ is to  be chosen later. Then we have proved  that
\begin{equation}\label{A(y)}
A(y) \ll y H^{-1}+H^{\frac12} y^{\frac1{2c}}+H^{-1} y^{1-\frac1c} \log H \ll y^{\frac1c-v}+y^{\frac{1+v} 2} .
\end{equation}

To bound $B(y)$, we put
$
f(m)= m^d\theta+h m^\frac1c .
$
Then for $m \sim y$,
$$
\left|f^{(d+1)}(m)\right| \asymp|h| y^{\frac1c-d-1}=\lambda.
$$
Applying Lemma \ref{estimation-1} to the inner sum of $B(y),$ we obtain
\begin{align*}
B(y) & \ll y^{\frac1c-1} \sum_{ 1\leqslant|h| \leqslant H}y^{1+\varepsilon}\left((|h|y^{\frac1c-d-1})^{\frac{1}{d(d+1)}}+y^{\frac{-1}{d(d+1)}}+y^{\frac{-2}{d(d+1)}}(|h| y^{\frac1c-d-1})^{\frac{-2}{d(d+1)^{2}}}\right) \\
& \ll
y^{(1+v)(1+\frac{1}{d(d+1)})-\frac1d+\varepsilon}+y^{1+v-\frac{1}{d(d+1)}+\varepsilon}+y^{(1+v)(1-\frac{2}{d(d+1)^{2}})+\varepsilon}.
\end{align*}
Choosing 
$$
v=\frac{1}{2}(1+\frac1{c})\cdot\frac{d(d+1)^2}{d(d+1)^{2}-1}-1,
$$ 
we get 
\begin{align*}
A(y)+B(y)&\ll y^{(1+v)(1-\frac{2}{d(d+1)^{2}})+\varepsilon}= y^{\frac{c+1}{2c}(1-h(d))+\varepsilon}
= y^{\frac{\theta(d,c)}{c}+\varepsilon}.
\end{align*}
Note that $v>0$ implies 
$$
c<\frac{d(d+1)^2}{d(d+1)^2-2}=1+\frac{2h(d)}{1-h(d)}=1+c_3(d).
$$
This proves the  desired estimate \eqref{main-estimate} for $4\le d\le 11$.

Assume $d \geqslant 12.$ We will use Lemma \ref{estimation-2} to bound $B(y)$. Let  $y> 1$ and  put
$$
z=y\left(|h| y^\frac1c\right)^{-\frac1{d_0+1}},
$$
where $d_0 \geqslant d+1$ is an integer to be chosen later. For each positive integer $n$ with $n \leqslant z$, one has
$$
\sum_{y<m \leqslant y^{\prime}} e(f(m))=\sum_{y<m \leqslant y^{\prime}} e(f(m+n))+O(z) .
$$
Summing over $n \in[1, z]$, we yield
$$
\sum_{y<m \leqslant y^{\prime}} e(f(m)) \ll \frac{1}{z} \sum_{y<m \leqslant y^{\prime}}\left|\sum_{n \leqslant z} e(f(m+n))\right|+z .
$$
Let $R_j(t)=(1+t)^\frac1c-F_j(t)$, where $F_j(t)=\sum_{0 \leqslant i \leqslant j}\left(\begin{array}{l}c^{-1} \\ i\end{array}\right) t^i$ is the $j$th Taylor polynomial of $(1+t)^\frac1c$.  
Taking $t=\frac{n}{m}$, we get
$$
f(m+n)  =(m+n)^d\theta+h m^\frac1c\left(F_{d_0}(\frac{n}{m})+R_{d_0}(\frac{n}{m})\right)  =P_{d_0}(n)+h m^\frac1c R_{d_0}(\frac{n}{m}),
$$
where $P_{d_0}(t)$ is a polynomial of degree $d_0$ whose $(d+1)$-th coefficient is $a_{d+1}$, i.e.
\begin{equation*}
a_{d+1}=h m^{\frac1c-d-1}\left(\begin{array}{c}
c^{-1} \\
d+1
\end{array}\right) ,
\end{equation*}
where
$
|h|\le H=y^{1-\frac1c+v}.
$
Noting that $R_{d_0}'(t) \ll|t|^{d_0}$ uniformly for $|t| \leqslant \frac{z}{y}$, we derive by partial integration that
$$
\sum_{n \leqslant z} e(f(m+n)) \ll(1+\underbrace{|h| y^\frac1c(\frac{z}{y})^{d_0+1}}_{\leqslant 1}) \max _{z_1 \leqslant z}\left|\sum_{1 \leqslant n \leqslant z_1} e\left(P_{d_0}(n)\right)\right| .
$$
By conjugating the last sum above if necessary, one can assume that $a_{d+1}>0$. Note that $a_{d+1}$ has the rational approximation $\left|a_{d+1}-\frac1{q}\right| \leqslant q^{-2}$, where $q=\left\lfloor \frac1{a_{d+1}}\right\rfloor\ge1$ since $a_{d+1}<1$. Furthermore, for $m \sim y$, we have $q \asymp y^{d+1-\frac1c}|h|^{-1}.$ 
Thus, by  applying Lemma \ref{estimation-2}, we derive that  for  $d_0\ge d+1$
\begin{equation*}
\sum_{1 \leqslant n \leqslant z_1} e\left(P_{d_0}(n)\right) 
\ll z_1^{1+\varepsilon}\left(q^{-1}+z_1^{-1}+q z_1^{-(d+1)}\right)^{\frac1{d_0(d_0-1)}}
\ll z^{1+\varepsilon}\left(|h| y^\frac1c\right)^{-v_0},
\end{equation*}
 where
\begin{align*}
v_0=\frac{d_0-d}{d_0(d_0^2-1)}.
\end{align*}
This  gives
\begin{equation*}
\sum_{y<m \leqslant y^{\prime}} e(f(m)) \ll z+y^{1+\varepsilon}\left(|h| y^\frac1c\right)^{-v_0} \ll y^{1+\varepsilon}\left(|h| y^\frac1c\right)^{-v_0}.
\end{equation*}
Hence
\begin{equation*}
B(y)=y^{\frac1c-1} \sum_{1 \leqslant|h| \leqslant H} y^{1+\varepsilon}\left(|h| y^\frac1c\right)^{-v_0} \ll y^{(1+v)\left(1-v_0\right)+\varepsilon} .
\end{equation*}
Choosing
\begin{equation*}
d_0=d_0(d)= \begin{cases}\frac{3d}2, & {\rm  if }\ 2 \mid d, \\
\frac{3 d-1}2, & {\rm  if }\ 2 \nmid d
\end{cases}
\end{equation*}
and
$$
v=\frac{\frac1c-1+v_0}{2- v_0},
$$
then picking up \eqref{A(y)}, we get
\begin{equation*}
A(y)+B(y) \ll y^{\frac1c-v}+y^{\frac{1+v} 2}+y^{(1+v)\left(1-v_0\right)+\varepsilon}
\ll y^{(1+\frac1c)\frac{1-v_0}{2-v_{0}}+\varepsilon}= y^{\frac{\theta(d,c)}{c}+\varepsilon}.
\end{equation*}
Note that $v>0$ means  
$$
c<\frac{1}{1-v_0}=1+\frac{v_0}{1-v_0}=1+c_3(d).
$$
This proves \eqref{main-estimate} for $d\ge 12$.
\end{proof}

\section{ The proof of Lemma \ref{psiupsilon}}\label{Restriction}
In this section, we give the proof of Lemma \ref{psiupsilon}. 
The proof is similar to the proof of \cite[Proposition 5.1]{Zhang-Zhang-2023}. 
\begin{lemma}\label{falphaeta}
    Let $\psi:\,\Z\to\C$ satisfy  $|\psi|\le\tau$, where
    \begin{align*}
\tau\left(n\right) =\begin{cases}\frac{c}{\sigma(b)}m^{d-\frac{1}{c}},&{\rm if~}Wn-b=m^d{\rm~for~some}\ m\in\mathbb N^c_x,\\
0,&\rm{otherwise}.
\end{cases}
\end{align*}
Then for any  $v> S(d)(1+\frac{2(c-1)}{1-\theta(d,c)})$  and $c\in\left(1,1+c_2(d)\right)$, we have
    \begin{align*}
    \int_0^1\left|\widehat{\psi}\left(\alpha\right)\right|^vd\alpha\ll_v N^{v-1}L^v,
    \end{align*}
where $S(d),$ $\theta(d,c)$ and $c_2(d)$ are  defined as in \rm{Lemma \ref{psiupsilon}.}
\end{lemma}

\begin{proof}
Write $u_1=\frac{2S(d)(c-1)}{1-\theta(d,c)}$ and 
$$
\cR_\delta:=\{\theta\in \T:\,|\widehat{\psi}(\theta)|>\delta NL\},\quad \delta\in (0,1).
$$
Following the argument as in the proof of \cite[Lemma 6.3]{Browning-Prendiville-2017}, it suffices to show that
\begin{equation}\label{mes-ineq}
\meas(\cR_\delta)\ll_{\varepsilon_1}\frac{1}{\delta^{S(d)+u_1+\varepsilon_1}N},
\end{equation}
where $\varepsilon_1>0$ is arbitrary.

By the definition of $\cR_{\delta}$, we have
\begin{align}\label{suppo1}
(\delta NL)^{S(d)}\cdot\meas(\cR_\delta)
&\le\int_{\T}|\widehat{\psi}(\theta)|^{S(d)}d\theta\le\int_{\T}|\widehat{\tau}(\theta)|^{S(d)}d\theta\nonumber\\
&\ll (x^{d-\frac{1}{c}})^{S(d)} \sum_{\substack{m_1,\dots,m_{S(d)} \le x \\
m_1^d +\dots+m_{{S(d)}/{2}}^d = m_{{S(d)}/{2}+1}^d+\dots+ m_{S(d)}^d}} 1\nonumber\\
&\ll x^{(d-\frac{1}{c})S(d)+S(d)-d}L\ll N^{S(d)(1+\frac{1}{d}(1-\frac{1}{c}))-1}L^{S(d)},
\end{align}
where we have used \eqref{1-zz} in the last inequality.

Assume $\delta\le N^{\frac{\theta(d,c)-1}{2dc}+\varepsilon_{3}}$ where $\varepsilon_{3}=\frac{S(d)(c-1)\varepsilon_{1}}{2dcu_{1}(u_{1}+\varepsilon_{1})}$.
Then  \eqref{suppo1} implies
$$
\meas(\cR_\delta)\ll \frac{1}{\delta^{S(d)}N^{1-\frac{S(d)(c-1)}{dc}}}\le
\frac1{\delta^{S(d)+u_1+\varepsilon_1} N}.
$$

Now assume  
$$ 
N^{\frac{\theta(d,c)-1}{2dc}+\varepsilon_{3}}<\delta <1.
$$
Let $\theta_1, \theta_2, \dots,\theta_R\in \cR_\delta$ be $N^{-1}$-spaced such that $\meas(\cR_{\delta}) \ll \frac{R}{N}$. Since $v>S(d)\ge 2d$, to prove \eqref{mes-ineq}, 
it suffices to show that
\begin{align}\label{R-condition}
R\ll_{\varepsilon_1}\frac{1}{\delta^{2d+\varepsilon_1}}.
\end{align}
For $\theta_r\in \cR_\delta$, we have $|\widehat{\psi}(\theta_r)| > \delta NL$ for $r=1,2, \dots, R$. It follows that
$$
R^2\delta^2N^2L^2\le \bigg(\sum_{r=1}^R|\widehat{\psi}(\theta_r)|\bigg)^2.
$$
Since $|\psi|\le \tau$, we can write $\psi(n)=a_n\tau(n)$ with $a_n\in\bC$ and $|a_n|\le 1$. Further, for $1\le r\le R$,  we write
$
|\widehat{\psi}(\theta_r)|=b_r\widehat{\psi}(\theta_r),
$
where $|b_r|=1$. Then by the Cauchy-Schwarz inequality, we have
\begin{align*}
R^2\delta^2N^2L^2\le&\bigg(\sum_{r=1}^Rb_r\sum_{n\in\Z} a_n\tau(n)e(\theta_rn)\bigg)^2 \notag \\
\le&\bigg(\sum_{n\in\Z}|a_n|^2\tau(n)\bigg)\bigg(\sum_{n\in\Z}\tau(n)\bigg|\sum_{r=1}^Rb_r  e(\theta_rn)\bigg|^2\bigg) \notag\\
\le&\bigg(\sum_{n\in\Z}\tau(n)\bigg)\sum_{1\le r,r'\le R}b_r\overline{b}_{r'}\sum_{n\in\Z}\tau(n)  e\big((\theta_r-\theta_{r'})n\big) \notag \\
\ll&NL
\sum_{1\le r,r'\le R}\big|\widehat{\tau}(\theta_r-\theta_{r'})\big|.
\end{align*}
Using H\"older's inequality, for $\kappa=d+\frac{\varepsilon_1}{3}$, we have
$$
R^2\delta^{2\kappa}N^{\kappa}L^{\kappa}\ll
\sum_{1\le r,r'\le R}\big|\widehat{\tau}(\theta_r-\theta_{r'})\big|^{\kappa}.
$$
By Lemma \ref{x831c} and noting that $c_2(d)<c_3(d),$ by similar argument as in \cite[Equation (6.2)]{Zhang-Zhang-2023}, we have, for $1<c<1+c_2(d),$  
\begin{equation*}\label{two}
\widehat{\tau}(\alpha)
=\widehat{\mu}\left(\alpha\right)+O\left(x^{d-\frac{1-\theta(d,c)}c+\varepsilon_{2}}\right),
\end{equation*}
where $\varepsilon_{2}=d\varepsilon_{3}$ and
\begin{equation*}
\mu\left(n\right) =\begin{cases}\frac{m^{d-1}}{\sigma(b)},&\text{if }Wn-b=m^d\ \text{for some}\ m\in \mathbb [x],\\
0,&\text{otherwise}.
\end{cases}
\end{equation*}
Therefore
\begin{equation}\label{nuOterm}
R^2\delta^{2\kappa}N^{\kappa}L^{\kappa}\ll
\sum_{1\le r,r'\le R}\big|\widehat{\mu}(\theta_r-\theta_{r'})\big|^{\kappa}+O\left(R^2 x^{\kappa(d-\frac{1-\theta(d,c)}c+\varepsilon_{2})}\right).
\end{equation}
By \eqref{Ddef}, $x^d\asymp NW,$ hence the $O$-term in \eqref{nuOterm} is $o(R^2\delta^{2\kappa}N^{\kappa}L^{\kappa})$.
Write
\begin{equation} \label{minor2}
\fn = \{ \alpha \in \bT: |\hat \mu (\alpha)| \le x^{d-\frac{\rho(d)}2} \},
\end{equation}
where 
\begin{equation*} 
\rho(d) =\begin{cases}2^{1-d}, \ &{\rm if}\ 2\le d\le 8,\\
\frac{1}{4(d^2-3d+3)},\ & {\rm if}\ d\ge 9.
\end{cases}
\end{equation*}
Following the arguments in the proof of  \cite[ Lemma 5.4]{Chow-2017}, we can obtain the following `major arc estimate': if $\alpha \in \bT \setminus \fn$, then there exist relatively prime integers $q$ and $a$ such that $0 \le a \le q-1$ and
\begin{equation} \label{major2}
\hat \mu (\alpha) \ll NLq^{\varepsilon-\frac1{d}}(1+N | \alpha - \frac{a}q |)^{-\frac1d}.
\end{equation}
Let us put $\theta_{r,r'}=\theta_{r}-\theta_{r'}$. By \eqref{minor2}, for $\theta_{r,r'}\in\fn$, we have
\begin{align*}
\sum_{\substack{1\le r,r'\le R\\
\theta_{r,r'}\in\fn}}\big|\widehat{\mu}(\theta_{r,r'})\big|^{\kappa}
&\ll R^2x^{\kappa(d-\frac{\rho(d)}2)}\ll R^2(WN)^{\kappa(1-\frac{\rho(d)}{2d})}\vspace{-0.2cm}\\[-0.7cm]
&\ll R^2N^{-\frac{\kappa\rho(d)}{2d}}N^{\kappa}L^{\frac{\kappa}2} =o(R^2\delta^{2\kappa}N^{\kappa}L^{\kappa}),
\end{align*}
where we have used the fact $\frac{1-\theta(d,c)}c<\frac{\rho(d)}2$ for $c\in(1,1+c_2(d))$.
Hence \eqref{nuOterm} becomes
\begin{equation} \label{Bourg4}
\del^{2 \kappa} N^\kappa L^{\kappa} R^2 \ll \sum_{\substack{1 \le r,r' \le R\\ \theta_{r,r'} \notin \fn}} |\hat \mu (\theta_{r,r'})|^\kappa.
\end{equation}

Now let $Q \asymp \del^{-3d}$. When $q > Q$, by \eqref{major2}, the  right hand side of \eqref{Bourg4} is
$
O(R^2 N^\kappa L^\kappa Q^{\varepsilon-\frac{\kappa}2}),
$
which is negligible comparing to the left hand side of \eqref{Bourg4}.
\ Hence the effective range for $q$ in \eqref{Bourg4} is $q\le Q.$ By \eqref{major2} we get 
\begin{equation*}
\del^{2 \kappa} R^2 \ll \sum_{q\le Q} \sum_{\substack{a~ {\rm mod}~q\\
(a, q)=1}}\sum_{1\le r,r'\le R}q^{\kappa\varepsilon-\frac\kappa{d}}(1+N | \theta_{r,r'} - \frac{a}q |)^{-\frac\kappa d}.
\end{equation*}
Hence
\begin{equation*}
\del^{2 \kappa} R^2 \ll \sum_{1 \le r,r' \le R} G_2(\theta_{r,r'}),
\end{equation*}
where
$$
G_2(\alp) = \sum_{q \le Q} \: \sum_{a=0}^{q-1}
\frac{q^{\kappa\varepsilon- \frac{\kappa}d}}{(1+N|\sin(\alp-\frac{a}{q})|)^{\frac{\kappa}{d}}}.
$$
Following the argument leading to \cite[Equation (4.16)]{Bourgain-1989}, but with $N^{2}$ replaced by $N$,
we get the desired bound \eqref{R-condition} for $R$, and hence finish the proof of Lemma \ref{falphaeta}.
\end{proof}

\begin{proof}[Proof of  Lemma \ref{psiupsilon}.] The proof follows step by step  the  proof of \cite[Proposition 5.1]{Zhang-Zhang-2023}, but replacing Lemma 5.2 in \cite[Section 5]{Zhang-Zhang-2023} by Lemma \ref{falphaeta} above.
 \end{proof}

\medskip

$\textbf{Acknowledgements}$ 
\bigskip

Part of this research was conducted while the second author was at Shandong University. The second author would like to thank  Bingrong Huang, Yongxiao Lin and  Lilu Zhao  for their warm hospitality. \\

This work is supported by National Natural Science Foundation of China (Grant No. 11871307 and Grant No. 12031008) and National Key Research and Development Program of China (Grant No. 2021YFA1000700). The second author was supported by UTUGS funding, working in the Academy of Finland project no. 333707.

\end{document}